\documentclass{amsart}
\usepackage{fullpage}
\usepackage{graphicx}
\usepackage{amsmath,amsfonts,amssymb,latexsym,epic,eepic}
\usepackage{dsfont}
\usepackage[mathcal]{euscript}
\usepackage{multicol}
\usepackage{tikz}
\usetikzlibrary{patterns}

\usepackage[colorlinks=true, pdfstartview=FitV, linkcolor=blue, citecolor=red, urlcolor=blue]{hyperref}

\newtheorem{theorem}{Theorem}[section]
\newtheorem{lemma}[theorem]{Lemma}

\newcommand{\bfe}{{\boldsymbol e}}
\newcommand{\bfei}{\bfe^{(i)}}

\newcommand{\bfn}{{\boldsymbol n}}
\newcommand{\bfu}{{\boldsymbol u}}

\newcommand{\bfv}{{\boldsymbol v}}

\newcommand{\bfx}{{\boldsymbol x}}

\newcommand{\bftau}{{\boldsymbol \tau}}

\newcommand{\Ent}[1]{{\lfloor #1 \rfloor}}
\newcommand{\dv}{\partial}
\newcommand{\lapi}{{\boldsymbol \Delta}}
\newcommand{\gradi}{{\boldsymbol \nabla}}
\newcommand{\dive}{{\rm div}}
\newcommand{\divv}{\boldsymbol{\rm div}}
\newcommand{\xN}{\mathbb{N}}
\newcommand{\xR}{\mathbb{R}}
\newcommand{\xL}{{\rm L}}
\newcommand{\xH}{{\rm H}}

\newcommand{\dx}{\,\mathrm{d}\bfx}

\newcommand{\norm}[1]{|\hspace{-.1em}| #1 |\hspace{-.1em}|}

\newcommand{\brok}{{1,\mesh}}
\newcommand{\edge}{\sigma}
\newcommand{\edges}{\mathcal{E}}
\newcommand{\edgesint}{{\mathcal E}_{{\rm int}}}
\newcommand{\edgesext}{{\mathcal E}_{{\rm ext}}}

\newcommand{\mesh}{{\mathcal M}}

\newcommand{\fluxK}{F_{K,\edge}}

\newcommand{\edged}{\varepsilon}
\newcommand{\edgeperp}{\tau}
\newcommand{\edgesd}{\bar {\mathcal E}}

\newcommand{\fluxd}{F_{\edge,\edged}}
\newcommand{\Ds}{{\scalebox{0.6}{$D_\edge$}}}
\newcommand{\ei}{^{(i)}}

\newcommand{\ie}{\emph{i.e.}}

\newcommand{\exm}{^{(m)}}
\newcommand{\ma}{\varepsilon}
\newcommand{\mcal}[1]{\mathcal{#1}}
\newcommand{\ds}{\displaystyle}

\newcommand{\deltap}{{\delta \hspace{-0.1em} p}}

\def\Ind{\mathds{1}}

\begin{document}
\title[Low Mach number limit of a pressure correction MAC scheme for barotropic flows]
{Low Mach number limit of a pressure correction MAC scheme for compressible barotropic flows}

\author{R. Herbin}
\address{I2M UMR 7373, Aix-Marseille Universit\'e, CNRS, \'Ecole Centrale de Marseille.}
\email{raphaele.herbin@univ-amu.fr}

\author{J.-C. Latch\'e}
\address{Institut de Radioprotection et de S\^uret\'e Nucl\'eaire (IRSN), Saint-Paul-lez-Durance, 13115, France.}
\email{jean-claude.latche@irsn.fr}

\author{K. Saleh}
\address{Universit\'e de Lyon, CNRS UMR 5208, Universit\'e Lyon 1, Institut Camille Jordan. 43 bd 11 novembre 1918; F-69622 Villeurbanne cedex, France.}
\email{saleh@math.univ-lyon1.fr}

\subjclass[2000]{35Q30,65N12,76M12}
\keywords{Compressible Navier-Stokes equations, low Mach number flows, finite volumes, MAC scheme, staggered discretizations.}

\begin{abstract}
We study the incompressible limit of a pressure correction MAC scheme \cite{her-14-cons} for the unstationary compressible barotropic Navier-Stokes equations.  Provided the initial data are well-prepared, the solution of the numerical scheme converges, as the Mach number tends to zero, towards the solution of the classical pressure correction \emph{inf-sup} stable MAC scheme for the incompressible Navier-Stokes equations.
\end{abstract}
\maketitle
%
%
\section{Introduction} \label{sec:intro}

Let $\Omega$ be parallelepiped of $\xR^d$, with $d\in\lbrace2,3\rbrace$ and $T>0$.  
The unsteady barotropic compressible Navier-Stokes equations, parametrized by the Mach number $\ma$, read for $(\bfx,t)\in\Omega\times (0,T)$: 
\begin{subequations}
\begin{align}
\label{eq:pb_mass} &\partial_t \rho^\ma + \dive( \rho^\ma\, \bfu^\ma) = 0, \displaybreak[1] \\
\label{eq:pb_mom}  & \partial_t (\rho^\ma\, \bfu^\ma) + \divv(\rho^\ma\, \bfu^\ma \otimes \bfu^\ma) 
 - \divv(\bftau(\bfu^\ma)) + \frac 1 {\ma^2}\ \gradi \wp(\rho^\ma) = 0, \displaybreak[1] \\ 
\label{eq:pb_CI}   &  \bfu^\ma|_{\dv\Omega}=0, \qquad  \rho^\ma|_{t=0} = \rho_0^\ma, \qquad \bfu^\ma|_{t=0}=\bfu_0^\ma,
\end{align} \label{eq:pb}\end{subequations}
where $\rho^\ma>0$ and $\bfu^\ma=(u_1^\ma,..,u_d^\ma)^T$ are the density and velocity of the fluid.
The pressure satisfies the ideal gas law $\wp(\rho^\ma)= (\rho^\ma)^\gamma$, with $\gamma \geq 1$, and  
\[
 \divv(\bftau(\bfu)) = \mu \Delta \bfu+ (\mu+\lambda)\gradi (\dive \,\bfu),
\]
where the real numbers $\mu$ and $\lambda$ satisfy $\mu>0$ and $\mu+\lambda>0$.
The smooth solutions of \eqref{eq:pb} are known to satisfy a kinetic energy balance and a renormalization identity.
In addition, under assumption on the initial data, it may be inferred from these estimates that the density $\rho^\ma$  tends to a constant $\bar \rho$, and the velocity tends, in a sense to be defined, to a solution $\bar u$ of the incompressible Navier-Stokes equations \cite{lio-98-inc}:
\begin{subequations}
\begin{align}
  &\dive \bar \bfu = 0, \\
  &\bar \rho \dv_t \bar \bfu + \bar \rho \divv(\bar \bfu \otimes \bar \bfu)- \mu \lapi \bar \bfu + \gradi \pi = 0,
\end{align}\label{eq:pb_inc}\end{subequations}
where $\pi$ is the formal limit of $(\wp(\rho^\ma)-\wp(\bar \rho))/\ma^2$.

\medskip
In this paper, we reproduce this theory for a pressure correction scheme, based on the Marker-And-Cell (MAC) space discretization: we first derive discrete analogues of the kinetic energy and renormalization identities, then establish from these relations that approximate solutions of \eqref{eq:pb} converge, as $\ma\to 0$, towards the solution of the classical projection scheme for the incompressible Navier-Stokes equations \eqref{eq:pb_inc}.

\medskip
For this asymptotic analysis, we assume that the initial data is ``well prepared": $\rho_0^\ma >0$, $\rho_0^\ma \in \xL^\infty(\Omega)$, $\bfu_0^\ma \in \xH_0^1(\Omega)^d$ and, taking without loss of generality $\bar \rho=1$, there exists $C$ independent of $\ma$ such that: 
\begin{equation} \label{eq:u0_rho0_wp}
\norm{\bfu_0^\ma}_{\xH^1(\Omega)^d}  +  \frac{1}{\ma}\, \norm{\dive \, \bfu_0^\ma}_{\xL^2(\Omega)}
+ \frac{1}{\ma^2}\,\norm{\rho^\ma_0- 1}_{\xL^\infty(\Omega)}  \leq  C.
\end{equation}
Consequently, $\rho_0^\ma$ tends to $1$ when $\ma\to 0$; moreover, we suppose that $\bfu_0^\ma$ converges in $\xL^2(\Omega)^d$ towards a function $\bar \bfu_0\in\xL^2(\Omega)^d$ (the uniform boundedness of the sequence in the $\xH^1(\Omega)^d$ norm already implies this convergence up to a subsequence).
%
%
\section{The numerical scheme}

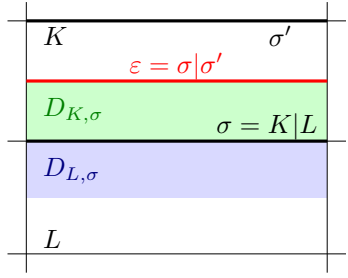
\begin{figure}[tb]
\begin{center}
\begin{tikzpicture}

\fill[blue!15!white] (0.5,1.25)--(4.5,1.25)--(4.5,2)--(0.5,2)--cycle;
\path node at (0.6,1.6) [anchor= west]{\textcolor{blue!50!black}{$D_{L,\edge}$}};
\fill[green!20!white] (0.5,2)--(4.5,2)--(4.5,2.8)--(0.5,2.8)--cycle;
\path node at (0.6,2.4) [anchor= west]{\textcolor{green!50!black}{$D_{K,\edge}$}};

\draw[very thin] (0.25,0.5)--(4.75,0.5);
\draw[very thin] (0.25,2.)--(4.75,2.);
\draw[very thin] (0.25,3.6)--(4.75,3.6);
\draw[very thin] (0.5,0.25)--(0.5,3.85);
\draw[very thin] (4.5,0.25)--(4.5,3.85);
\draw[very thick] (0.5,2.)--(4.5,2.); \path node at (3.7,2.2) {$\edge=K|L$};
\draw[very thick] (0.5,3.6)--(4.5,3.6); \path node at (3.88,3.4) {$\edge'$};
\path node at (0.6,3.4) [anchor= west]{$K$};
\path node at (0.6,0.7) [anchor= west]{$L$};

\draw[very thick, red] (0.5,2.8)--(4.5,2.8); \path node at (2.5,3) {\textcolor{red}{$\edged=\edge|\edge'$}};

\path node at (6,3.6) [anchor= west]{\begin{minipage}{0.4\textwidth} primal cells: $K$, $L$. \end{minipage}};
\path node at (6,2.9) [anchor= west]{\begin{minipage}{0.4\textwidth} dual cell for the $y$-component of the velocity: $D_\edge=D_{K,\edge}\cup D_{L,\edge}$. \end{minipage}};
\path node at (6,2) [anchor= west]{\begin{minipage}{0.4\textwidth} primal and dual cells $d$-dimensional measures: $|K|$, $|D_\edge|$, $|D_{K,\edge}|$. \end{minipage}};
\path node at (6,1.1) [anchor= west]{\begin{minipage}{0.4\textwidth} faces $(d-1)$-dimensional measures: $|\edge|$, $|\edged|$. \end{minipage}};
\path node at (6,0.3) [anchor= west]{\begin{minipage}{0.4\textwidth} vector normal to $\edge$ outward $K$: $\bfn_{K,\edge}$. \end{minipage}};
\end{tikzpicture}
\end{center}

\caption{Notations for control volumes and faces. \label{fig:mesh}}
\end{figure}

Let $\mesh$ be a MAC mesh (see e.g. \cite{ghlm-17-macns}  and Figure \ref{fig:mesh} for the notations).
The  discrete density unknowns are associated with the cells of the mesh $\mesh$, and are denoted by $\big\{\rho_K,\ K \in \mesh \big\}$.
We denote by $\edges$ the set of the faces of the mesh, and by $\edges\ei$ the subset of the faces orthogonal to the $i$-th vector of the canonical basis of $\xR^d$.
The discrete $i^{th}$ component of the velocity is located at the centre of the faces $\edge \in \edges \ei$, so the whole set of discrete velocity unknowns reads $\big\{ u_{\edge,i},\ \edge \in \edges\ei,  1 \leq i \leq d \big\}$.
We define $\edgesext=\{\edge\in\edges, \edge\subset \partial \Omega \}$, $\edgesint=\edges \setminus \edgesext$, $\edgesint\ei=\edgesint \cap  \edges\ei$ and $\edgesext\ei=\edgesext \cap  \edges\ei$.
The boundary conditions \eqref{eq:pb_CI} are taken into account by setting $u_{\edge,i}=0$ for all $\edge \in \edgesext\ei$, $1 \leq i \leq d$.
Let $\delta t>0$ be a constant time step. 
The approximate solution $(\rho^n,\bfu^n)$ at time $t_n=n\delta t$ for $1\leq n\leq N=\Ent{T/\delta t}$ is computed as follows: knowing $\{\rho_K^{n-1}, \rho_K^{n}, K\in\mesh\} \subset \xR$ and $(u_{\edge,i}^{n})_{\edge\in\edgesint\ei,1\leq i\leq d}\subset\xR$, find $(\rho_K^{n+1})_{K\in\mesh} \subset \xR$ and $(u_{\edge,i}^{n+1})_{\edge\in\edgesint\ei, 1\leq i\leq d}\subset\xR$ by the following algorithm:

\begin{subequations}\label{eq:correction_scheme}
\begin{align}
\nonumber &
\mbox{{\bf Pressure gradient scaling step}:}
\\ \label{eq:sch_scale} &  \displaystyle \hspace{5ex}
\mbox{For } 1 \leq i \leq d,\ \forall \edge \in \edgesint\ei, \quad
(\overline{\gradi p})^{n}_{\edge,i} = \Bigl(\frac{\rho^n_\Ds}{\rho^{n-1}_\Ds}\Bigr)^{1/2}  (\gradi p^n)_{\edge,i}.
\displaybreak[1]\\ \nonumber &
\mbox{{\bf Prediction step} -- Solve for $\tilde \bfu^{n+1}$:}
\\ \nonumber & \hspace{5ex}
\mbox{For } 1 \leq i \leq d,\ \forall \edge \in \edgesint\ei,
\\ \label{eq:sch_mom} & \hspace{5ex}
\dfrac 1 {\delta t} \bigl(\rho^n_\Ds \tilde u^{n+1}_{\edge,i}-\rho^{n-1}_\Ds u_{\edge,i}^n \bigr)
\!+ \!\dive(\rho^n \tilde u^{n+1}_i \bfu^n)_{\edge}
\!- \!\dive \bftau(\tilde \bfu^{n+1})_{\edge,i}
\!+ \!\dfrac{1}{\ma^2}(\overline {\gradi p})^{n}_{\edge,i}
\!=0.
\displaybreak[1] \\[0.5ex]  \nonumber &
\mbox{{\bf Correction step} -- Solve for $\rho^{n+1}$ and $\bfu^{n+1}$:}
\\ \nonumber & \hspace{5ex}
\mbox{For } 1 \leq i \leq d,\ 
 \forall \edge \in \edgesint\ei,
\\[0.5ex] \label{eq:sch_cor} & \displaystyle \phantom{\hspace{5ex} \forall K \in \mesh, \quad}
\dfrac 1 {\delta t}\ \rho^n_\Ds\ (u^{n+1}_{\edge,i}-\tilde u_{\edge,i}^{n+1})
+ \dfrac{1}{\ma^2}\,(\gradi p^{n+1})_{\edge,i}- \dfrac{1}{\ma^2}\,(\overline{\gradi p})_{\edge,i}^{n}
=0,
\\ \label{eq:sch_mass} & \hspace{5ex}
\forall K \in \mesh, \quad
\dfrac 1 {\delta t}(\rho^{n+1}_K-\rho^n_K) + \dive(\rho^{n+1} \bfu^{n+1})_K = 0,
\\ \label{eq:sch_eos} & \hspace{5ex}
\forall K \in \mesh, \quad p_K^{n+1}=\wp(\rho^{n+1}),
\end{align}
\end{subequations}
where the discrete densities and space operators are defined below (see also \cite{her-14-cons,gra-15-unc}).

\medskip
\noindent {\bf Mass convection flux --}
Given a discrete density field $\rho = \lbrace \rho_K, \, K\in\mesh\rbrace$, and a velocity field $\bfu=\lbrace u_{\edge,i}, \, \edge \in \edges\ei, \, 1\leq i \leq d\rbrace$,  the convection term  in \eqref{eq:sch_mass} reads:
\begin{equation}
\label{div_mass}
\dive(\rho \bfu)_K = \frac 1 {|K|} \sum_{\edge \in\edges(K)} \fluxK(\rho,\bfu), \qquad K\in\mesh,
\end{equation}
where $\fluxK(\rho,\bfu)$ stands for the mass flux across $\edge$ outward $K$.
This flux is set to 0 on external faces to account for the homogeneous Dirichlet boundary conditions; it is given on internal faces by:
\begin{equation}\label{eq:def_FKedge}
\fluxK(\rho,\bfu)= |\edge|\ \rho_\edge\ u_{K,\edge}, \qquad  \edge\in\edgesint,\, \edge=K|L,
\end{equation}
where $u_{K,\edge}  = u_{\edge,i}\ \bfn_{K,\edge} \cdot \bfe\ei$, with $\bfe^{(i)}$ the $i$-th vector of the orthonormal basis of $\xR^d$.
The density at the face $\edge=K|L$ is approximated by the upwind technique, \ie\ $\rho_\edge=\rho_K$ if $u_{K,\edge} \geq 0$ and $\rho_\edge=\rho_L$ otherwise.

\medskip
\noindent {\bf Pressure gradient term --}
In \eqref{eq:sch_scale} and \eqref{eq:sch_cor}, the term $(\gradi p)_{\edge,i}$ stands for the $i^{th}$ component of the discrete pressure gradient at the face $\edge$.
Given a discrete density field $\rho = \lbrace \rho_K, \, K\in\mesh\rbrace$, this term is defined as:
\begin{equation}\label{eq:def_grad_p}
(\gradi p)_{\edge,i}=\frac{|\edge|}{|D_\edge|}  (\wp(\rho_L)-\wp(\rho_K))\ \bfn_{K,\edge} \cdot \bfe\ei,
\quad 1\leq i \leq d,\ \edge\in\edgesint\ei,\, \edge=K|L.
\end{equation}
Defining for all $K\in\mesh$, $(\dive \bfu)_K = \dive (1\times \bfu)_K$ (see \eqref{div_mass}), the following discrete duality relation holds for all discrete density and velocity fields $(\rho,\bfu)$:
\begin{equation} \label{eq:grad-div}
\sum_{K \in \mesh} |K|\wp(\rho_K)\ (\dive \bfu)_K
+\sum_{i=1}^d \sum_{\edge\in\edgesint\ei} |D_\edge|\ u_{\edge,i}\ (\gradi p)_{\edge,i}
=0.
\end{equation}
The MAC scheme is \textit{inf-sup} stable: there exists $\beta>0$, depending only on $\Omega$ and the regularity of the mesh, such that, for all $p=\lbrace p_K,\,K\in\mesh\rbrace$, there exists $\bfu = \lbrace u_{\edge,i}, \, \edge \in \edges\ei, \, 1\leq i \leq d\rbrace$ satisfying homogeneous Dirichlet boundary conditions with:
\[
\norm{\bfu}_\brok=1 \text{ and } \sum_{K\in\mesh} |K|\, p_K\,(\dive \bfu )_K \geq \beta\, \norm{p-\frac 1 {|\Omega|} \int_\Omega p \dx}_{L^2(\Omega)},
\]
where $\norm{\bfu}_\brok$ is the usual discrete $\xH^1$-norm of $\bfu$ (see \cite{ghlm-17-macns}).

\medskip
\noindent {\bf Velocity convection operator --} Given a density field $\rho = \lbrace \rho_K, \, K\in\mesh\rbrace$, and two velocity fields $\bfu=\lbrace u_{\edge,i}, \, \edge \in \edges\ei, \, 1\leq i \leq d\rbrace$ and $\bfv=\lbrace v_{\edge,i}, \, \edge \in \edges\ei, \, 1\leq i \leq d\rbrace$, we build for each $\edge\in\edgesint$ the following quantities:
\begin{itemize}
\item an approximation of the density on the dual cell $\rho_\Ds$ defined as:
\begin{equation}\label{eq:def_rho}
|D_\edge|\, \rho_\Ds = |D_{K,\edge}|\, \rho_K + |D_{L,\edge}|\, \rho_L, \qquad \edge\in\edgesint,\, \edge=K|L,
\end{equation}

\item a discrete divergence for the convection on the dual cell $D_\edge$:
\[
\dive(\rho v_i \bfu)_{\edge}= \sum_{\edged\in\edgesd(D_\edge)} \fluxd(\rho,\bfu)\ v_{i,\edged}, \qquad \edge \in \edgesint\ei, \, 1\leq i \leq d.
\]
For $i\in \{1,..,d \}$, and $\edge \in \edgesint \ei$, $\edge=K|L$,
\begin{list}{-}{\itemsep=0.ex \topsep=0.5ex \leftmargin=1.cm \labelwidth=0.7cm \labelsep=0.3cm \itemindent=0.cm}
\item  If the vector $\bfei$ is normal to $\edged$, $\edged$ is included in a primal cell $K$, and we denote by $\edge'$ the second face of $K$ which, in addition to $\edge$, is normal to $\bfei$.
We thus have $\edged=D_\edge | D_{\edge'}$.
Then the mass flux through $\edged$ is given by:
\begin{equation}\label{eq:flux_eK}
F_{\edge,\edged}(\rho, \bfu)  = \frac 1 2  \bigl( F_{K,\edge}(\rho,\bfu)\ \bfn_{D_\edge,\edged} \cdot \bfn_{K,\edge}
+ F_{K,\edge'}(\rho,\bfu)\ \bfn_{D_\edge,\edged} \cdot \bfn_{K,\edge'} \bigr).
\end{equation}
\item  If the vector $\bfei$ is tangent to $\edged$, $\edged$ is the union of the halves of two primal faces $\edgeperp$ and $\edgeperp'$ such that $\edgeperp\in \edges(K)$ and $\edgeperp' \in \edges(L)$.
The mass flux through $\edged$ is then given by:
\begin{equation}\label{eq:flux_eorth}
F_{\edge,\edged}(\rho, \bfu)  = \frac 1 2\ \bigl(F_{K,\edgeperp}(\rho,\bfu)+ F_{L,\edgeperp'}(\rho,\bfu) \bigr).
\end{equation}
\end{list}
\end{itemize}
With this definition, the dual fluxes are locally conservative through dual faces $\edged=D_\edge|D_{\edge'}$ (\ie\ $F_{\edge,\edged}(\rho, \bfu)=-F_{\edge',\edged}(\rho, \bfu)$), and vanish through a dual face included in the boundary of $\Omega$.
For this reason, the values $v_{\edged,i}$ are only needed at the internal dual faces, and are chosen centered, \ie, for $\edged=D_\edge|D_{\edge'}$, $v_{\edged,i}=(v_{\edge,i}+v_{\edge',i})/2$.

\medskip
As a result, a finite volume discretization of the mass balance \eqref{eq:pb_mass} holds over the internal dual cells. Indeed, if $ \rho^{n+1} = \lbrace \rho_K^{n+1}, \, K\in\mesh\rbrace$, $\rho^{n} = \lbrace \rho_K^{n}, \, K\in\mesh\rbrace$ and $\bfu^{n+1} = \lbrace u_{\edge,i}^{n+1}, \, \edge \in \edges\ei, \, 1\leq i \leq d\rbrace$ are density and velocity fields satisfying  \eqref{eq:sch_mass}, then, the dual quantities $\lbrace \rho_\Ds^{n+1}, \rho_\Ds^{n}, \, \edge \in \edgesint\rbrace$ and the dual fluxes $\lbrace \fluxd(\rho^{n+1},\bfu^{n+1}), \, \edge \in \edgesint, \, \edged \in \edgesd(D_\edge) \rbrace$ satisfy a finite volume discretization of the mass balance \eqref{eq:pb_mass} over the internal dual cells:
\begin{equation}\label{eq:mass_D}
\frac{|D_\edge|}{\delta t}  (\rho^{n+1}_\Ds-\rho^{n}_\Ds)
+ \sum_{\edged\in\edgesd(D_\edge)} \fluxd(\rho^{n+1},\bfu^{n+1})=0, \qquad  \edge\in\edgesint.
\end{equation}

\noindent {\bf Diffusion term --} The discrete diffusion term in \eqref{eq:sch_mom} is defined in \cite{gra-15-unc} and is coercive in the following sense: for every discrete velocity field $\bfu$ satisfying the homogeneous Dirichlet boundary conditions, one has:
\begin{equation}
\label{eq:diff_coerc}
-\sum_{i=1}^d \sum_{\edges \in \edgesint\ei} |D_\edge|\ u_{\edge,i}\ \dive \bftau(\bfu)_{\edge,i} \geq  \mu\ \norm{\bfu}_\brok^2.
\end{equation}

\medskip
The initialization of the scheme \eqref{eq:correction_scheme} is performed by setting 
\[
\displaystyle \forall K \in \mesh, \rho_K^{0} = \frac 1 {|K|} \int_K \rho_0^\ma(\bfx) \dx \mbox{ and } \forall \edge \in \edgesint\ei,  1\leq i\leq d,  \ u_{\edge,i}^0 = \frac 1 {|\edge|} \int_{\edge} \bfu_0^\ma(\bfx) \cdot \bfe\ei\dx,
\]
and computing $\rho^{-1}$ by solving the backward mass balance equation \eqref{eq:sch_mass} for $n=-1$ where the unknown is $\rho^{-1}$ and not $\rho^0$. This allows to perform the first prediction step with $\lbrace \rho_\Ds^{0}, \rho_\Ds^{-1}, \, \edge \in \edgesint\rbrace$ and the dual mass fluxes $\lbrace \fluxd(\rho^{0},\bfu^{0}), \, \edge \in \edgesint, \, \edged \in \edgesd(D_\edge) \rbrace$ satisfying the mass balance \eqref{eq:mass_D}.
Moreover, since $\rho_0^\ma>0$, one clearly has $\rho_K^{0} > 0$ for all $K\in\mesh$ and therefore $\rho_\Ds^{0}>0$ for all $\edge \in \edgesint$. The positivity of $\rho^{-1}$ is a consequence of the following Lemma.

\begin{lemma}
If  $(\rho_0^\ma,\bfu_0^\ma)$ satisfies \eqref{eq:u0_rho0_wp}, then there exists $C$, depending on the mesh but independent of $\ma$ such that:
\begin{equation} \label{eq:sch:init}
\frac{1}{\ma^2} \, \max\limits_{K\in\mesh}|\rho_K^0-1|  + \frac{1}{\ma^2} \, \max\limits_{1\leq i \leq d} \max\limits_{\edge\in\edgesint\ei}  \,|(\gradi p)^0_{\edge,i}| + \frac{1}{\ma} \, \max\limits_{K\in\mesh}|\rho_K^{-1}-1| \leq C.
\end{equation}
\end{lemma}

\begin{proof}
We sketch the proof. The boundedness of the first two terms is a straightforward consequence of \eqref{eq:u0_rho0_wp}. For the third term we remark that, again by \eqref{eq:u0_rho0_wp}:
 \[
\forall K \in \mesh, \quad  \rho_K^{-1}-1 = \underbrace{\rho_K^{0}-1}_{= \mcal{O}(\ma^2)} \ + \ \underbrace{\delta t \, \rho_K^0 (\dive \bfu^0)_K}_{= \mcal{O}(\ma)} \ + \ \underbrace{\delta t \sum_{\edge\in\edges(K)} \frac{|\edge|}{|K|} (\rho^0_\edge - \rho_K^{0}) \bfu_{K,\edge}^0}_{= \mcal{O}(\ma^2)}.
 \]
 \end{proof}

%
%
\section{Asymptotic analysis of the zero Mach limit} \label{sec:cscheme}

By the results of \cite{her-14-cons},  there exists a solution $(\rho^n,\bfu^n)_{0\leq n\leq N}$ to the scheme \eqref{eq:correction_scheme} and any solution satisfies the following relations:

\textbullet \, a discrete kinetic energy balance: \
for all $\edge \in \edgesint\ei$,  $1 \leq i \leq d$, $0\leq n \leq N-1$:
\begin{multline}
\label{kin-eq}
\dfrac 1 {2\delta t} \Bigl(\rho_\Ds^{n}\,|u_{\edge,i}^{n+1}|^2 - \rho_\Ds^{n-1}\,|u_{\edge,i}^{n}|^2 \Bigr) 
+ \frac{1}{2|D_\edge|} \sum_{\substack{\edged \in \edgesd(D_\edge)\\ \edged=D_\edge|D_{\edge'}}} \hspace{-2ex}
\fluxd(\rho^{n},\bfu^{n}) \,\tilde u_{\edge,i}^{n+1}\,\tilde u_{\edge',i}^{n+1}
\\
- \dive \bftau(\tilde \bfu^{n+1})_{\edge,i}\, \tilde u_{\edge,i}^{n+1}
+ \frac{1}{\ma^2}(\gradi p)_{\edge,i}^{n+1} \, u_{\edge,i}^{n+1}
+ \frac{\delta t}{\ma^4}  \Bigl( \frac{| (\gradi p)^{n+1}_{\edge,i}|^2}{2\,\rho^n_\Ds} -\frac{| (\gradi p)^{n}_{\edge,i} |^2}{2\,\rho^{n-1}_\Ds}\Bigr)
\\
+ R_{\edge,i}^{n+1}=0, \qquad \mbox{with } R_{\edge,i}^{n+1}= \dfrac{1}{2\delta t} \rho_\Ds^{n-1}(\tilde u_{\edge,i}^{n+1}-u_{\edge,i}^{n})^2.
\end{multline}

\textbullet \, a discrete renormalization identity: \ for all $K\in\mesh$, $0\leq n \leq N-1$:
\begin{equation}
\label{renorm-eq}
\dfrac 1 {\delta t} \Bigl( \Pi_\gamma(\rho^{n+1}_K)-\Pi_\gamma(\rho^n_K) \Bigr)
+ \dive \bigl(b_\gamma(\rho^{n+1}) \bfu^{n+1} -b_\gamma'(1)\rho^{n+1}\bfu^{n+1} \bigr)_K   +  p_K^{n+1} \, \dive( \bfu^{n+1})_K + R_K^{n+1} = 0,
\end{equation}
with $R_K^{n+1}\geq 0$, where the function $b_\gamma$ is defined by $b_\gamma(\rho)=\rho\log\rho$ if  $\gamma=1$, $b_\gamma(\rho)=\rho^\gamma/(\gamma-1)$ if $\gamma>1$ and satisfies $\rho b_\gamma'(\rho)-b_\gamma(\rho)=\rho^\gamma= \wp(\rho)$ for all $\rho>0$, and $\Pi_\gamma(\rho)=b_\gamma(\rho)-b_\gamma(1) - b_{\gamma}'(1)(\rho-1)$.

\medskip
Summing \eqref{kin-eq} and \eqref{renorm-eq} over the primal cells from one side, and over the dual cells and the components on the other side, and invoking the grad-div duality relation \eqref{eq:grad-div}, we obtain a local-in-time discrete entropy inequality, for $0\leq n \leq N-1$:
\begin{multline}\label{sch_loc_ener}
\frac 1 2 \!\sum_{i =1}^d \!\! \sum_{\edge\in \edgesint\ei} \!|D_\edge| \!\Big (\rho^{n}_\Ds\ |u^{n+1}_{\edge,i}|^2 \!-  \rho^{n-1}_\Ds\ |u^n_{\edge,i}|^2 \Big ) \! +\! \frac 1 {\ma^2}\! \! \sum_{K\in \mesh}\!\!|K| \! \Big ( \Pi_\gamma(\rho^{n+1}_K) - \Pi_\gamma(\rho^{n}_K) \Big )
\\
\!+\!\mu  \delta t  \norm{\tilde \bfu^{n+1}}_{1,\mesh}^2 \!+ \!\frac 1 {\ma^4} \!\sum_{i =1}^d \!\!\sum_{\edge\in \edgesint\ei} \!\! |D_\edge| \delta t^2   \Bigl(\!\frac{| (\gradi p)^{n+1}_{\edge,i} |^2}{2\,\rho^n_\Ds}   -\frac{| (\gradi p)^{n}_{\edge,i} |^2}{2\,\rho^{n-1}_\Ds}\!\Bigr) \!\!+ \!\mathcal{R}^{n+1}\!\! \leq 0
\end{multline}
where $\displaystyle \mathcal{R}^{n+1} =\sum_{i=1}^d \sum_{\edge \in \edgesint\ei} R^{n+1}_{\edge,i} + \frac 1 {\ma^2} \sum_{K \in \mesh} R_K^{n+1}\geq 0$.

\medskip
The function $\Pi_\gamma$ has some important properties:
\begin{subequations}
 \begin{align}
    & \label{upperPi} \bullet \ \mbox{For all $\gamma \geq 1$ there exists $C_\gamma$ such that:} \  \Pi_\gamma(\rho) \leq C_\gamma \,|\rho-1|^2,\ \forall\rho\in (0,2). \displaybreak[1] \\
   & \label{lowerPi1} \bullet \  \mbox{If $\gamma \geq 2$ then} \ \Pi_\gamma(\rho) \geq |\rho-1|^2, \ \forall \rho>0. \displaybreak[1] \\
   &  \label{lowerPi2} \begin{array}{l}
      \bullet \ \mbox{If $\gamma \in [1,2)$ then for all $R\in(2,+\infty)$, there exists $C_{\gamma,R}$ such that:} \displaybreak[1] \\
      \hspace{2cm} \begin{array}{ll}
      \Pi_\gamma(\rho) \geq C_{\gamma,R}   |\rho-1|^2,     & \ \forall\rho\in (0,R), \\ 
      \Pi_\gamma(\rho) \geq C_{\gamma,R}   |\rho-1|^\gamma, & \ \forall  \rho\in [R,\infty).
      \end{array} 
      \end{array}
\end{align}
\end{subequations}

\begin{lemma}[Global discrete entropy inequality]
\label{lmm:global_entropy}
Under assumption \eqref{eq:u0_rho0_wp}, there exists $C_0>0$ independent of $\ma$ such that the solution $(\rho^n,\bfu^n)_{0\leq n\leq N}$ to the scheme \eqref{eq:correction_scheme} satisfies, for $\ma$ small enough, and for $1\leq n\leq N$:
\begin{multline} \label{eq:corr_stap_bis}\hspace{7ex}
\frac 1 2 \sum_{i =1}^d \sum_{\edge\in \edgesint\ei} |D_\edge| \rho^{n-1}_\Ds\ |u^n_{\edge,i}|^2
+\mu\  \sum_{k=1}^n \delta t\ \norm{\tilde \bfu^k}_{1,\mesh}^2
\\  
+\frac 1 {\ma^2} \sum_{K\in \mesh}|K| \, \Pi_\gamma(\rho^n_K)
+ \frac 1 {\ma^4}  \sum_{i =1}^d \sum_{\edge\in\edgesint\ei} \frac{|D_\edge|\ \delta t^2}{2\,\rho^{n-1}_\Ds}\ |(\gradi p)^n_{\edge,i}|^2
\leq C_0.
\hspace{7ex}\end{multline}
\end{lemma}

\begin{proof} 
Summing \eqref{sch_loc_ener} over $n$ yields the inequality \eqref{eq:corr_stap_bis} with
\begin{equation}
C_0= \frac 1 2 \sum_{i =1}^d \sum_{\edge\in \edgesint\ei} |D_\edge|  \rho^{-1}_\Ds\ |u^0_{\edge,i}|^2
+\frac 1 {\ma^2} \sum_{K\in \mesh}|K| \, \Pi_\gamma(\rho^0_K)
+ \frac 1 {\ma^4} \sum_{i =1}^d \sum_{\edge\in\edgesint\ei} \frac{|D_\edge|\ \delta t^2}{2\,\rho^{-1}_\Ds}\ |(\gradi p)^0_{\edge,i}|^2. 
\end{equation}
By \eqref{eq:sch:init}, for $\ma$ small enough, one has $\rho_K^{-1}\leq 2$ for all $K\in\mesh$ and therefore $\rho_\Ds^{-1}\leq 2$ for all $\edge\in\edgesint\ei$ and $1\leq i\leq d$. 
Hence, since $\bfu_0^\ma$ is uniformly bounded in $\xH^1(\Omega)^d$ by \eqref{eq:u0_rho0_wp}, a classical trace inequality yields the boundedness of the first term. Again by \eqref{eq:sch:init}, one has $|\rho_K^{0}-1|\leq C\ma^2$ for all $K\in\mesh$. 
Hence, by \eqref{upperPi}, the second term vanishes as $\ma\to 0$. The third term is also uniformly bounded with respect to $\ma$ thanks to \eqref{eq:sch:init}.
\end{proof}

\begin{lemma}[Control of the pressure]
\label{lmm:sch:controle:p}
Assume  that $(\rho_0^\ma,\bfu_0^\ma)$ satisfies \eqref{eq:u0_rho0_wp} and let $(\rho^n,\bfu^n)_{0\leq n\leq N}$ satisfy \eqref{eq:correction_scheme}. Let  $p^n=\wp(\rho^n)$ and define $\deltap^n = \lbrace \deltap_K^n,\,K\in\mesh \rbrace $ where $\deltap_K^n = (p_K^n-|\Omega|^{-1}\int_\Omega p^n \dx)/\ma^2$. Then, one has, for all $1\leq n \leq N$:
\[
\norm{\deltap^n} \leq C_{\mesh,\delta t},
\]
where  $C_{\mesh,\delta t} \ge 0$ depends on the mesh and $\delta t$ but not on $\ma$, and $\norm{\cdot}$ stands for any norm on the space of discrete functions.
\end{lemma}

\begin{proof}
By \eqref{eq:corr_stap_bis}, the discrete pressure gradient is controlled in $\xL^\infty$ by $C_{\mesh,\delta t}\, \ma^2$, so that $\gradi(\delta p^n)$ is bounded in any norm independently of $\ma$. 
Using the discrete $(\xH^{-1})^d$-norm (see e.g. \cite{ghlm-17-macns}), invoking the gradient divergence duality \eqref{eq:grad-div} and the \emph{inf-sup} stability of the scheme, $\norm{\gradi(\delta p^n)}_{-1,\mesh} \leq C_{\mesh,\delta t}$ implies that $\norm{\deltap^n}_{\xL^2} \leq \beta^{-1} C_{\mesh,\delta t}$.
\end{proof}
 
\begin{theorem}[Incompressible limit of the MAC pressure correction scheme]\ \\
Let $(\ma\exm)_{m\in \xN}$ be a sequence of positive real numbers tending to zero, and let $(\rho\exm,\bfu\exm)$ be a corresponding sequence of solutions of the scheme \eqref{eq:correction_scheme}.
Then the sequence $(\rho\exm)_{m \in \xN}$ converges to the constant function $\rho =1$ when $m$ tends to $+ \infty$ in $\xL^\infty((0,T),\xL^q(\Omega))$, for all $q \in [1,\min(\gamma,2)]$.

In addition, the sequence $(\bfu\exm,\deltap\exm)_{m \in \xN}$ tends, in any discrete norm, to the solution $(\bfu,\deltap)$ of the usual MAC pressure correction scheme for the incompressible Navier-Stokes equations, which reads:
\begin{subequations}\label{eq:limit_correction_scheme}
\begin{align*}
\nonumber &
\mbox{{\bf Prediction step} -- Solve for $\tilde \bfu^{n+1}$:}
\\ \nonumber & \qquad
\mbox{For } 1 \leq i \leq d,\ \forall \edge \in \edgesint\ei,
\qquad
\dfrac 1 {\delta t}\ \bigl(\tilde u^{n+1}_{\edge,i}- u_{\edge,i}^n \bigr)
+ \dive(\tilde u^{n+1}_i \bfu^n)_\edge
- \dive \bftau(\tilde \bfu^{n+1})_{\edge,i}
+ (\gradi (\delta p)^n)_{\edge,i}
=0.
\\[2ex] \nonumber &
\mbox{{\bf Correction step} -- Solve for $(\delta p)^{n+1}$ and $\bfu^{n+1}$:}
\\ \label{eq:limit_sch_cor} & \qquad
\mbox{For } 1 \leq i \leq d,\ 
 \forall \edge \in \edgesint\ei,
\qquad
\dfrac 1 {\delta t}\ (u^{n+1}_{\edge,i}-\tilde u_{\edge,i}^{n+1})
+ (\gradi (\delta p)^{n+1})_{\edge,i}- (\gradi (\delta p)^n)_{\edge,i}
=0,
\\ & \qquad
\forall K \in \mesh, \quad
\hspace{19ex}
\dive(\bfu^{n+1})_K = 0.
\end{align*}
\end{subequations}
\end{theorem}

\begin{proof}
By \eqref{lowerPi1} and the global entropy estimate \eqref{eq:corr_stap_bis}, one has for $\gamma\geq 2$, 
\[
\norm{\rho\exm(t)-1}_{\xL^2(\Omega)}^2 \leq  \int_\Omega \Pi_\gamma(\rho\exm(t))\leq \,C_0\, \ma^2, \, \forall t\in(0,T).
\]
For $1\leq\gamma \leq 2$, invoking \eqref{lowerPi2} and estimate \eqref{eq:corr_stap_bis}, we obtain for all $t\in(0,T)$ and for all $R\in(2,+\infty)$:
\[
\begin{array}{ll}
(i) & \ds \quad  \norm{(\rho\exm(t)-1)\Ind_{\{ \rho^{\exm}(t)\leq R\}}}_{\xL^2(\Omega)}^2 \leq
\, \frac{1}{C_{\gamma,R}} \int_\Omega \Pi_\gamma(\rho^{\exm}(t))  \leq C \, \ma^2,  \forall t\in(0,T),
\\
(ii) & \ds \quad \norm{(\rho^{\exm}(t)-1)\Ind_{\{ \rho^{\exm}(t)\geq R \}} }_{\xL^\gamma(\Omega)}^\gamma \leq
\, \frac{1}{C_{\gamma,R}} \int_\Omega \Pi_\gamma(\rho^{\exm}(t))  \leq \, C \, \ma^2, \forall t\in(0,T),
\end{array}
\]
which proves the convergence of $(\rho\exm)_{m \in \xN}$ to the constant function $\rho =1$ as $m \to + \infty$ in $\xL^\infty((0,T),\xL^q(\Omega))$ for all $q \in [1,\min(\gamma,2)]$.
Using again \eqref{eq:corr_stap_bis},  the sequence $(\bfu\exm)_{m \in \xN}$ is bounded in any discrete norm and the same holds for the sequence $(\deltap\exm)_{m \in \xN}$ by Lemma \ref{lmm:sch:controle:p}. By the Bolzano-Weiertrass theorem and a norm equivalence argument, there exists a subsequence of  $(\bfu\exm,\deltap\exm)_{m \in \xN}$ which tends, in any discrete norm, to a limit $(\bfu,\deltap)$. Passing to the limit cell-by-cell in \eqref{eq:correction_scheme}, one obtains that $(\bfu,\deltap)$ is a solution to \eqref{eq:limit_correction_scheme}.
Since this solution is unique, the whole sequence converges, which concludes the proof.
\end{proof}
%
%
\bibliographystyle{plain}
\bibliography{Low_Mach_MAC_HLS}
\end{document}